\documentclass[12pt]{amsart}
\usepackage{amssymb}

\evensidemargin\paperwidth
\advance\evensidemargin-\textwidth
\advance\oddsidemargin-1in
\evensidemargin\oddsidemargin

\topmargin\paperheight
\advance\topmargin-\textheight
\advance\topmargin-1in

\theoremstyle{plain}
\newtheorem{theorem}{Theorem}[section]
\theoremstyle{plain}
\newtheorem{proposition}[theorem]{Proposition}
\theoremstyle{plain}
\newtheorem{lemma}[theorem]{Lemma}
\theoremstyle{plain}
\newtheorem{corollary}[theorem]{Corollary}
\theoremstyle{plain}
\newtheorem{definition}[theorem]{Definition}

\theoremstyle{plain}

\theoremstyle{remark}
\newtheorem{remark}[theorem]{Remark}
\theoremstyle{remark}

\theoremstyle{remark}
\newtheorem{acknowledgment}{Acknowledgment}

\begin{document}

\title
{The class $\mathcal{S}$ as an ME invariant}

\author{Hiroki Sako}
\address{Department of Mathematical Sciences, University of Tokyo, Komaba, Tokyo, 153-8914, Japan.}
\email
{hiroki@ms.u-tokyo.ac.jp}

\begin{abstract}
We prove that being in Ozawa's class $\mathcal{S}$
is a measure equivalence invariant.
\end{abstract}

\keywords{measurable group theory; solid von Neumann algebra.}

\maketitle

\section{introduction}

Measurable group theory is a new field and has been attracting many researchers having various backgrounds. The discipline deals with how much information on countable groups is preserved through measure equivalence. The notion of measure equivalence was given by Gromov \cite{gromov} as a variant of quasi-isometry. Two groups are said to be \textit{measure equivalent} (ME) if there exists an ME coupling, instead of topological coupling.

In most cases, much information is lost through ME couplings. For example, any two countable amenable groups are ME. This is a consequence of \cite{Ornstein--Weiss}, \cite{Connes--Feldman--Weiss}. Many people are interested in finding small measure equivalence classes (higher rank lattices \cite{Furman; Higher Rank Lattice}, mapping class groups with high complexity \cite{Kida}) or in classifying non-amenable groups up to measure equivalence.

We will prove that Ozawa's class $\mathcal{S}$ defined in \cite{Ozawa; Kurosh} is an ME invariant class. The class was defined by means of topological amenability \cite{Anantharaman-Delaroche} on the largest boundary. Ozawa and Popa proved classification results on group von Neumann algebras of the class $\mathcal{S}$ (\cite{ozawa; solid}, \cite{Ozawa--Popa; Prime Factorization}).

\begin{definition}[\cite{Ozawa; Kurosh}]\label{Definition; The class S}
A countable group $G$ is said to be in $\mathcal{S}$ if the left-times-right
translation action of $G \times G$ on $\beta G \cap G^c$ is amenable,
where $\beta G \cap G^c$ is the Gelfand spectrum of the commutative $C^*$-algebra
$\ell_\infty G / c_0 G$.
\end{definition}

The following is the main theorem of this paper.

\begin{theorem}\label{Introduction; S is MEinvariant}
If $G$ and $\Gamma$ are ME and if $\Gamma \in \mathcal{S}$,
then $G \in \mathcal{S}$.
\end{theorem}

The class $\mathcal{S}$ is an intermediate class between the set of exact groups and that of amenable groups, which are also ME invariant classes. These three classes are also characterized by topological amenability. A countable group is exact if and only if there exists an amenable action on a compact space \cite{ozawa; amenable actions and exactness}. A countable group is amenable if and only if any continuous action on any compact space is amenable.

By Hjorth's theorem \cite{hjorth; cost}, A countable group $G$ is treeable in the sense of Pemantle and Peres \cite{pemantle--peres; treeable}, if and only if $G$ is ME to a free group ($\mathbb{Z}, \mathbb{F}_2$ or $\mathbb{F}_\infty$). As a corollary of Theorem~\ref{Introduction; S is MEinvariant}, the class of treeable groups is an intermediate ME invariant class between $\mathcal{S}$ and the set of amenable groups. We get the following fact on group von Neumann algebras:

\begin{corollary}
If $G$ is ME to a free group, then the group von Neumann algebra $L(G)$ is solid, namely, every diffuse subalgebra has the injective relative commutant.
\end{corollary}

Since the free groups are in the class $\mathcal{S}$, if $G$ is ME to a free group, then $\Gamma \in \mathcal{S}$. By \cite{ozawa; solid}, $L(G)$ is solid.

\section{Measure Equivalence and Measure Embedding}
The following is a generalization of Gromov's measure equivalence ($0.5.$E in \cite{gromov}).

\begin{definition}\label{definition; measure embedding}
Let $G$ and $\Gamma$ be countable groups.
We say that the group $G$ {\rm measurably embeds into} $\Gamma$,
if there exist a standard measure space $(\Sigma, \nu)$, a measure preserving
action of $G \times \Gamma$ on $\Sigma$ and measurable subsets $X, Y \subset \Sigma$
with the following properties:
\begin{eqnarray*}
    \Sigma = \bigsqcup_{\gamma \in \Gamma} \gamma (X)
    = \bigsqcup_{g \in G} g (Y), \quad
    \nu(X) < \infty.
\end{eqnarray*}
Then we use the notation $G \preceq_{\rm ME} \Gamma$.
The infinite measure space $\Sigma$ equipped with the $G \times \Gamma$-action
is called a {\rm measure embedding of} $G$ {\rm into} $\Gamma$.
The measure embedding $\Sigma$ is said to be {\rm ergodic},
if the $G \times \Gamma$-action is ergodic.
\end{definition}

If the subset $Y$ also has finite measure, then $\Sigma$ gives an \textit{ME coupling} between $G$ and $\Gamma$ and these groups are said to be \textit{measure equivalent} (ME). As in the case of an ME coupling, if there exists a measure embedding of $G$ into $\Gamma$, there exists an ergodic one by using ergodic decomposition. See Lemma~2.2 in Furman \cite{Furman; Higher Rank Lattice} for the proof.

\begin{remark}\label{Remark; MEm}
\begin{enumerate}
\item
    The relation $\preceq_\mathrm{ME}$ is transitive;
    if $H \preceq_\mathrm{ME} \Lambda$ and $\Lambda \preceq_\mathrm{ME} \Gamma$, then $H \preceq_\mathrm{ME} \Gamma$. The proof is same as that of transitivity of measure equivalence.
\item
    If countable groups $G$ and $\Gamma$ satisfy $G \preceq_\mathrm{ME} \Gamma$ and if $\Gamma$ is amenable, then $G$ is also amenable. Exactness has the same property
    $($see Remark~$\ref{Remark; Exactness is an ME invariant})$.
\end{enumerate}
\end{remark}

\section{The class $\mathcal{S}$ as an ME Invariant}
\label{Section; The class S as an ME Invariant}

Theorem \ref{Theorem; S is MEinvariant} is stronger than Theorem~\ref{Introduction; S is MEinvariant} and follows from Proposition \ref{Proposition; S is MEinvariant in SOE setting}.

\begin{theorem}\label{Theorem; S is MEinvariant}
If $G \preceq_{\mathrm{ME}} \Gamma$ and $\Gamma \in \mathcal{S}$,
then $G \in \mathcal{S}$.
\end{theorem}

\begin{proposition}\label{Proposition; S is MEinvariant in SOE setting}
Suppose $\Gamma \in \mathcal{S}$.
Let $\beta$ be a free measure preserving $($m.p.$)$ action of $\Gamma$ on a standard measure space $(Y, \mu)$ and let $\alpha$ be a free m.p.\ action of $G$
on a measurable subset $X \subset Y$ with measure $1$.
If their orbits satisfy
$\alpha(G) (x) \subset \beta(\Gamma) (x)$ for ${\rm a.e. \ } x \in X$,
then $G \in \mathcal{S}$.
\end{proposition}

In this proposition, $Y$ can be an infinite standard measure space.

\begin{proof}[Proof of Theorem~\ref{Theorem; S is MEinvariant}
from Proposition~\ref{Proposition; S is MEinvariant in SOE setting}]
Suppose that $(\Sigma, \nu)$ gives an ergodic measure embedding of $G$ into $\Gamma$. Choose $G$ fundamental domain $Y$ and $\Gamma$ fundamental domain $X$.
We can replace $\Sigma$ so that the $G$-action on
$\Gamma \backslash \Sigma \cong X$ is free, by taking a product of $\Sigma$ and a $G$-probability space on which free m.p.~$G$-action is given.
Then we can consider $\Sigma$ comes from a stable orbit equivalence with
constant $s = \nu(Y) / \nu(X) \in (0, \infty]$.
This argument is covered by Lemma 3.2 in \cite{Furman; OE rigidity}
and Remark 2.14 in \cite{Monod--Shalom} if $s < \infty$.
We do the same argument in the case of $s = \infty$.

We note that $\Gamma \in \mathcal{S}$ if and only if $\Gamma \times \mathbb{Z} / n \mathbb{Z} \in \mathcal{S}$. Replacing $\Gamma$ with $\Gamma \times \mathbb{Z} / n \mathbb{Z}$, if necessary, we get
$\mathcal{R}(G \curvearrowright X)^s \cong \mathcal{R}(\Gamma \curvearrowright Y)$
with constant $1 \le s$, where $\mathcal{R}(G \curvearrowright X),  \mathcal{R}(\Gamma \curvearrowright Y)$ mean the orbit equivalence relations of
free m.p.\ actions of $G$ and $\Gamma$.
By Proposition~\ref{Proposition; S is MEinvariant in SOE setting}, $\Gamma \in \mathcal{S}$ implies $G \in \mathcal{S}$.
\end{proof}

To prove Proposition~\ref{Proposition; S is MEinvariant in SOE setting}, we use the notation given as follows.

We introduce a measure $\nu$ on $\mathcal{R}_\beta$ as the push forward under the map
\begin{eqnarray*}
    Y \times \Gamma \ni (y, \gamma) \mapsto (\alpha(\gamma)(y), y) \in \mathcal{R}_\beta,
\end{eqnarray*}
where the measure on $Y \times \Gamma$ is given by the product of $\mu$ and the counting measure.
The measure $\nu$ coincides with the measure defined in Feldman and Moore \cite{Feldman--Moore; II}.

The action $\beta$ (resp.~$\alpha$) gives a group action of $\Gamma$ (resp.~$G$)
on $L^\infty(Y)$ (resp.~$L^\infty(X)$).
We use the same notation $\beta$ (resp.~$\alpha$) for this action. Let $p \in L^\infty(Y)$ be
the characteristic function of $X$. The algebra $L^\infty(Y)$ and the group $\Gamma$ are represented on $L^2(\mathcal{R}_\beta, \nu)$ as
\begin{eqnarray*}
    ( F \xi ) (x, y) &=& F(x) \xi(x, y), \quad F \in L^\infty(Y),\\
    ( u_\gamma \xi ) (x, y) &=& \xi(\beta(\gamma^{-1})(x), y),
     \quad \gamma \in \Gamma, \xi \in L^2(\mathcal{R}_\beta),
     (x, y) \in \mathcal{R}_\beta.
\end{eqnarray*}
Let $B$ be the $C^\ast$-algebra generated by the images, which is the reduced crossed product algebra
$B = L^\infty(Y) \rtimes_{\rm red} \Gamma$.
Its weak closure is the group measure space construction
$\mathcal{M} = L^\infty(Y) \rtimes \Gamma$ \cite{Murray--vN; IV}.
We denote by $\mathrm{tr}$ the canonical faithful normal semi-finite trace on $\mathcal{M}$ with normalization $\mathrm{tr}(p) = 1$.
The unitary involution $J$ of $(\mathcal{M}, \mathrm{tr})$ is written as
$(J \xi ) (x,y) = \overline{\xi(y, x)}$, $(\xi \in L^2(\mathcal{R}_\beta)$, $(x, y) \in \mathcal{R}_\beta)$.

The group $G$ is represented on
$p L^2(\mathcal{R}_\beta) = L^2(\mathcal{R}_\beta \cap (X \times Y))$ by
\begin{eqnarray*}
    ( v_g \xi ) (x, y) &=& \xi(\alpha(g^{-1})(x), y),
                \quad g \in G, \xi \in p L^2(\mathcal{R}_\beta),
                (x, y) \in \mathcal{R}_\beta \cap (X \times Y).
\end{eqnarray*}
We denote by $C^*_\lambda (G)$ the $C^\ast$-algebra generated by these operators.
The algebra is isomorphic to the reduced group $C^*$-algebra of $G$.
The Hilbert space $L^2(\mathcal{R}_\alpha, \nu)$ can be
identified with a closed subspace of $p L^2(\mathcal{R}_\beta)$.
The algebra $C^*_\lambda (G)$ is also represented on $L^2(\mathcal{R}_\alpha)$ faithfully.
We denote by $P$ the orthogonal projection from $L^2(\mathcal{R}_\beta)$ onto $L^2(\mathcal{R}_\alpha)$.
We note that the algebra $p B p$ does not contain $C^*_\lambda (G)$ in general,
although there exists an inclusion between their weak closures.

Let $e_\Delta$ be the projection from $L^2(\mathcal{R}_\beta)$ onto the set of $L^2$-functions
supported on the diagonal subset of $\mathcal{R}_\beta$. This is the Jones projection for $L^\infty(Y) \subset \mathcal{M}$.
For $\gamma \in \Gamma$ and a finite subset $\Gamma_0 \subset \Gamma$, we define the projections
$e(\gamma), e(\Gamma_0)$ by
\begin{eqnarray*}
    e(\gamma) = J u_\gamma J e_\Delta J u_\gamma^* J
              , \quad
    e(\Gamma_0) = \sum_{\gamma \in \Gamma_0} e(\gamma).
\end{eqnarray*}
For $g \in G$ and a finite subset $G_0 \subset G$, we define the projections $f(g), f(G_0)$ by
\begin{eqnarray*}
    f(g) = v_g e_\Delta v_g^* = v_g p e_\Delta v_g^*
         , \quad
    f(G_0) = \sum_{g \in G_0} f(g).
\end{eqnarray*}

Let $K \subset \mathcal{B}(L^2(\mathcal{R}_\beta))$ be the hereditary subalgebra of $\mathcal{B}(L^2(\mathcal{R}_\beta))$
with approximate units $\{ e(\Gamma_0) \ | \ \Gamma_0 \subset \Gamma {\rm \ finite } \}$,
that is,
\begin{eqnarray*}
    K =
    \overline{\bigcup_{\Gamma_0} e(\Gamma_0) \mathcal{B}(L^2(\mathcal{R}_\beta)) e(\Gamma_0)}^{\ \| \cdot \|}.
\end{eqnarray*}
It is easy to see that $B$ and $J B J$ are in the multiplier of $K$, so is $D = C^\ast(B, JBJ)$.

The algebra $B$ satisfies the following continuity property, which is similar to Proposition 4.2 in \cite{Ozawa; Kurosh}.

\begin{proposition}\label{Proposition; ME Continuity for Gamma}
The $*$-homomorphism
\begin{eqnarray*}
\widetilde{\Psi} \colon B \otimes_{\mathbb{C}} J B J \rightarrow (D + K) / K
\end{eqnarray*}
given by $\widetilde{\Psi}(b \otimes c) = bc + K$ is continuous with respect to the minimal tensor norm.
\end{proposition}

For the rest of this paper, $\otimes$ stands for the minimal tensor product between two $C^*$-algebras.

\begin{proof}
Consider the representation of $\ell_\infty \Gamma$ on $L^2(\mathcal{R}_\beta)$ given by
\begin{eqnarray*}
    (m_{\phi} \xi)(\gamma y, y) = \phi(\gamma) \xi (\gamma y, y), \quad \phi \in \ell_\infty \Gamma, \xi \in L^2(\mathcal{R}_\beta), \gamma \in \Gamma, y \in Y.
\end{eqnarray*}
Let $\widetilde{D}$ be the $C^*$-algebra generated by $D$ and $m(\ell_\infty \Gamma)$. The algebra $\widetilde{D}$ is also in the multiplier of $K$.
Since preimage $m^{-1}(K \cap \mathrm{image}(m))$ is $c_0 \Gamma$, $m$ also gives an embedding of $\ell_\infty \Gamma / c_0 \Gamma$ into $(\widetilde{D} + K )/ K$. The embedding $m$ and the representation of $L^\infty Y$ give a  $*$-homomorphism
\begin{eqnarray*}
     \widetilde{\Psi} \colon E = \ell_\infty \Gamma / c_0 \Gamma \otimes L^\infty Y \otimes J L^\infty Y J
    \longrightarrow (\widetilde{D} + K )/ K.
\end{eqnarray*}
Here, we used the fact that abelian $C^\ast$-algebras are nuclear (\cite{Takesaki}). Consider the action of $\Gamma \times \Gamma$ on $E$ given by \begin{eqnarray*}
    \mathfrak{A}(\gamma_1, \gamma_2)((\phi + c_0 \Gamma) \otimes f_1 \otimes J f_2 J) = (l_\gamma r_\lambda \phi +c_0 \Gamma) \otimes \beta(\gamma_1) (f_1) \otimes J \beta(\gamma_2)(f_2) J,
\end{eqnarray*}
where $l_\cdot$ and $r_\cdot$ stand for the left and the right translations.
Since $\ell_\infty \Gamma / c_0 \Gamma$ is in the center of $E$, by \cite{Anantharaman-Delaroche}, the full crossed product coincides with the reduced crossed product $F = E \rtimes_{\mathrm{red}} (\Gamma \times \Gamma)$ and this is nuclear.
The unitary representations $u_\cdot$ and $J u_\cdot J$ give a $*$-homomorphism
$\widetilde{\Psi} \colon F
    \rightarrow (\widetilde{D} + K )/ K$.
By restricting $\widetilde{\Psi}$ on $(L^\infty Y \otimes J L^\infty Y J) \rtimes_{\mathrm{red}} (\Gamma \times \Gamma)$, we get a map satisfying Proposition \ref{Proposition; ME Continuity for Gamma}.
\end{proof}

If the $\ast$-homomorphism $\Psi \colon B \otimes_{\mathbb{C}} J B J$ given by $\Psi(b \otimes c) = bc$ is continuous with respect to the minimal tensor norm, then the group $\Gamma$ is amenable. The above proposition can be regarded as
a weakened amenability property for the $\Gamma$-action.

We make use of the following characterization
of $\mathcal{S}$. Proposition 15.2.3 and a variant of Lemma 15.1.4 in \cite{Brown--Ozawa; Approximation} imply the following.

\begin{proposition}\label{Proposition; Continuity for G}
A countable group $G$ is in $\mathcal{S}$
if and only if $G$ is exact and there exists a contractive c.p.~map
$\Phi \colon C^\ast_{\lambda}(G) \otimes C^\ast_{\rho}(G)
\rightarrow \mathcal{B}(\ell_2 G)$,
satisfying
\begin{eqnarray*}
\Phi(b \otimes c) - bc \in \mathcal{K}(\ell_2 G), \quad b \in C^\ast_{\lambda}(G), c \in C^\ast_{\rho}(G),
\end{eqnarray*}
where $C^\ast_{\lambda}(G)$ and $C^\ast_{\rho}(G)$ are
the $C^\ast$-algebras
generated by the left and right regular representations, respectively.
\end{proposition}

\begin{remark}
\label{Remark; Exactness is an ME invariant}
By using the notion of weak exactness introduced in Kirchberg \cite{Kirchberg; Weak Exactness},
we get that the exactness on $\Gamma$ implies that on $G$.
Indeed, the algebra $\mathcal{M}$ is weakly exact by the exactness of $\Gamma$.
The subalgebra $L(G) \subset p \mathcal{M} p$ is also weakly exact.
Thus $G$ is exact by Ozawa's theorem \cite{Ozawa; Weakly Exact von Neumann algebras}.
\end{remark}

We have only to show the existence of $\Phi$ in Proposition
\ref{Proposition; Continuity for G}.
However, lack of inclusion ``$C^*_\lambda (G) \subset B$'' requires some technical elaboration.
\begin{lemma}\label{Lemma; ME Projections}
There exists a sequence $\{ q_n \}_{n = 1, 2, \cdots}$ of projections in $L^\infty(X)$
satisfying:
\begin{enumerate}
    \item
        The sequence $\{ q_n \}$ is increasing and strongly converges to $p$;
    \item
        For any finite subset $G_0 \subset G$ and $n$, there exists a finite subset $\Gamma_0 \subset \Gamma$
        satisfying
            $q_n f(G_0) \le e(\Gamma_0)$;
    \item
        For any finite subset $\Gamma_0 \subset \Gamma$ and $n$, there exists a finite subset $G_0 \subset G$
        satisfying
            $q_n e(\Gamma_0) P \le f(G_0)$.
\end{enumerate}
\end{lemma}
We note that projections $P, p, e(\gamma), f(g), (\gamma \in \Gamma, g \in G)$ are
in the commutative von Neumann algebra $L^\infty(\mathcal{R}_\beta) \subset \mathcal{B}(L^2(\mathcal{R}_\beta))$.
Every projection in $L^\infty(\mathcal{R}_\beta)$ which is less than $f(g) \ (resp.\ pe(\gamma))$ is
of the form $qf(g) \ (resp.\ q e(\gamma))$ for some $q \in L^\infty(X)$.
We also note that  $q f(g) \le e(\Gamma_0)$ if and only if there exists
a partition $\{ q_\gamma \in L^\infty(X) \ | \ \gamma \in \Gamma_0 \}$ of $q$
such that $q v_g  = \sum_{\gamma \in \Gamma_0} q_\gamma u_\gamma $.

\begin{proof}
We fix an index on $G$; $\{g_1, g_2, \ldots\} = G$. For any $g_k$,
the projection $e(\Gamma_0) f(g_k)$ can be written as $Q(g_k, \Gamma_0) f(g_k)$
by some projection $Q(g_k, \Gamma_0) \in L^\infty(X)$.
The net of projections
$\{ e(\Gamma_0) f(g_k) \ | \ \Gamma_0 \subset \Gamma {\rm \ finite}\}$
strongly converges to $f(g_k)$. For any natural number $n$,
there exists a finite subset $\Gamma_{k, n} \subset \Gamma$
such that $\mathrm{tr}(Q(g_k, \Gamma_{k, n})) \ge 1 - 2^{-(n + k)}$.
Then the projections $Q_n = \bigwedge_{k = 1}^\infty Q(g_k, \Gamma_{k, n})$ satisfy $\mathrm{tr}(Q_n) \ge 1 - 2^{-n}$ and
\begin{eqnarray*}
    Q_n f(g_k) \le Q(g_k, \Gamma_{k, n}) f(g_k) = e(\Gamma_{k, n}) f(g_k) \le e(\Gamma_{k, n}).
\end{eqnarray*}
Let $\{ q_n \}$ be the increasing sequence of projections $\{ \bigvee_{l = 1}^n Q_l \}$. Then we have $\mathrm{tr}(q_n) \ge 1 -2^{-n}$ and
\begin{eqnarray*}
    q_n f(g_k) \le e \left( \bigcup_{l = 1}^n \Gamma_{k,l} \right), \quad g_k \in G.
\end{eqnarray*}
It turned out that the sequence $\{q_n\}$ satisfies $(1)$ and $(2)$.

By a similar technique, we get a sequence $\{ p_n \}$ with $(1)$ and $(3)$, since
\begin{eqnarray*}
    {\rm str}\lim_{G_0} f(G_0) e(\gamma) = e(\gamma) P, \quad \gamma \in \Gamma.
\end{eqnarray*}
Taking products $\{p_n q_n\}$, we get a sequence which satisfies $(1)$, $(2)$ and $(3)$ at the same time.
\end{proof}

The Hilbert space $\ell_2 G$ embeds into $L^2(\mathcal{R}_\alpha)$ by the isometry
\begin{eqnarray*}
    \ell_2 G \ni \delta_g \mapsto v_g \xi_\Delta = J v_g^* J \xi_\Delta \in L^2(\mathcal{R}_\alpha),
\end{eqnarray*}
where the $L^2$-function $\xi_\Delta$ is the characteristic function of the diagonal subset of $\mathcal{R}_\alpha$.
We regard $\ell_2 G$ as a subspace of $L^2(\mathcal{R}_\alpha)$ by this map.
The subspace $\ell_2 G$ is invariant under the action of $C^*_\lambda G$ and $J C^*_\lambda G J$.

\begin{lemma}\label{Lemma; ME Cutting}
For a projection $q \in L^\infty(X)$,
the following inequality on operator norm holds true: $\| (1 - q J q J) |_{\ell_2 G} \| \le (2 - 2\mathrm{tr}(q))^{1 / 2}$.
\end{lemma}
\begin{proof}
It suffices to show $\| \eta - q J q J \eta \|^2 \le (2 - 2 \mathrm{tr}(q)) \| \eta \|^2$ for any vector $\eta \in \ell_2 G$.
Since $q J q J P = P q J q J$ and $\eta = P \eta$, we get
\begin{eqnarray*}
         \| \eta - q J q J \eta \|^2
     &=& \sum_{g \in G} \| f(g) (\eta - q J q J \eta) \|^2
      =  \sum_{g \in G} \| f(g) \eta - q J q J f(g) \eta \|^2,\\
         \| \eta \|^2
     &=& \sum_{g \in G} \| f(g) \eta \|^2,
\end{eqnarray*}
The claim reduces to the inequality
$\| f(g) \eta - q J q J f(g) \eta \|^2 \le (2 - 2 \mathrm{tr}(q)) \| f(g) \eta \|^2$.

We note that $\eta$ takes a constant value $\eta(g)$ on the set $\{(\alpha(g)(x), x) \in \mathcal{R}_\alpha \ | \ x \in X \}$.
By a direct computation, we get
\begin{eqnarray*}
    \| f(g) \eta \|^2
                      = \int_{\mathcal{R}_{\alpha}} |(f(g)\eta) (y, x)|^2 d \nu
                      = \int_{x \in X} |\eta (\alpha(g)(x), x)|^2 d \mu
                      = |\eta(g)|^2.
\end{eqnarray*}
Let $X_q \subset X$ be a measurable subset such that $\chi (X_q) = p - q$.
The measure of subset $X_0 = \{x \in X \ | \ x \in X_q {\rm \ or\ } \alpha(g)(x) \in X_q\}$
satisfies $\nu(X_0) \le 2 - 2 \mathrm{tr}(q)$. Then we get
\begin{eqnarray*}
      \| f(g) \eta - q J q J f(g) \eta \|^2
    = \int_{x \in X_0} |\eta (\alpha(g)(x), x)|^2 d \mu
    = \mu(X_0) |\eta(g)|^2.
\end{eqnarray*}
Our claim was confirmed.
\end{proof}

We finish the proof of Proposition~\ref{Proposition; S is MEinvariant in SOE setting}.

\begin{proof}
[Proposition~\ref{Proposition; S is MEinvariant in SOE setting}]
We will show the existence of $\Phi$ in Proposition~\ref{Proposition; Continuity for G}.
We consider that $C^\ast_{\lambda}(G)$ is a subalgebra of $\mathcal{B}(p L^2(\mathcal{R}_\beta))$
and $C^\ast_{\rho}(G)$ is $J C^\ast_{\lambda}(G) J$.
By $P_0$ we denote the orthogonal projection from $L^2(\mathcal{R}_\beta)$ onto $\ell_2 G$.

Let $\{q_n\}$ be the sequence satisfying Lemma \ref{Lemma; ME Projections}.
and let $B_0$ be the $C^*$-algebra generated by $p$ and $\bigcup_n q_n C^*_\lambda (G) q_n$. The condition $(2)$ in Lemma \ref{Lemma; ME Projections} means $B_0 \subset B$. We recall that $B_0 \otimes J B_0 J$ is separable and that $F$ in the proof of Proposition \ref{Proposition; ME Continuity for Gamma} is nuclear.
By Choi--Effros lifting theorem \cite{Choi--Effros}, there exists a contractive c.p.~lifting $\Psi \colon B_0 \otimes J B_0 J \rightarrow D + K$ for
$\widetilde{\Psi} |_{B_0 \otimes J B_0 J}$.
We define contractive c.p.~maps $\Phi_n \colon C^\ast_{\lambda}(G) \otimes C^\ast_{\rho}(G) \rightarrow \mathcal{B}(\ell_2 G)$ by
\begin{eqnarray*}
\Phi_n(b \otimes J c J) = P_0 Q_n \Psi(q_n  b q_n \otimes J q_n c q_n J) Q_n P_0,
\end{eqnarray*}
where $Q_n = q_n J q_n J$. By the condition (3) in Lemma \ref{Lemma; ME Projections}, we get $P_0 Q_n K Q_n P_0 \subset \mathcal{K}(\ell_2 G)$. The element $\Phi_n(b \otimes J c J)$ is in
\begin{eqnarray*}
P_0 Q_n (b J c J + K) Q_n P_0 \subset P_0 Q_n b J c J Q_n P_0 + \mathcal{K}(\ell_2 G).
\end{eqnarray*}
The sequence $\{P_0 Q_n b J c J Q_n P_0 + \mathcal{K}(\ell_2 G)\} \subset \mathcal{B}(\ell_2 G) / \mathcal{K}(\ell_2 G)$ converges to $P_0 b J c J P_0 + \mathcal{K}(\ell_2 G)$, by the inequality
\begin{eqnarray*}
    & &     \| P_0 b J c J P_0 - P_0 Q_n b J c J Q_n P_0 \| \\
    &\le&   \| P_0 (1 - Q_n) b J c J P_0 \|
          + \| P_0 Q_n b J c J (1 - Q_n) P_0 \|\\
    &\le&   2 \| (1 - Q_n) P_0 \| \|b\| \|c\| \\
    &\le&  2 (2 - 2 \mathrm{tr}(q_n))^{1/2}\|b\| \|c\|.
\end{eqnarray*}
It follows that the natural $*$-homomorphism from the minimal tensor product $\widetilde{\Phi} \colon C^\ast_{\lambda}(G) \otimes C^\ast_{\rho}(G) \rightarrow \mathcal{B}(\ell_2 G) / \mathcal{K}(\ell_2 G)$ is given and $\widetilde{\Phi}$ is a limit of liftable maps. By Theorem $6$ of \cite{Arveson}, there exists a contractive c.p.~lifting $\Phi$ for $\widetilde\Phi$.
\end{proof}

\section{Final remark}
As a consequence of Proposition~\ref{Proposition; S is MEinvariant in SOE setting}, we get the following indecomposability of an equivalence relation given by a class $\mathcal{S}$ group.

\begin{corollary}\label{Corollary; Indecomposability}
Let $\Gamma$ be a countable group and let $H \subset G$ be an inclusion of countable groups. Suppose that $\Gamma \in \mathcal{S}$ and that the centralizer $Z_G (H)$ is non-amenable.
Let $\beta$ be a free m.p.\ action of $\Gamma$ on a standard measure space $(Y, \mu)$ and let $\alpha$ be a free m.p.\ action of $G$
on a measurable subset $X \subset Y$ with measure $1$.
If the orbits satisfy
$\alpha(G) (x) \subset \beta(\Gamma) (x)$ for ${\rm a.e. \ } x \in X$,
then $H$ is finite.
\end{corollary}

\begin{proof}
The class $\mathcal{S}$ has the following property: If $G \in \mathcal{S}$ and $Z_G (H)$ is non-amenable, then $H$ is finite.
\end{proof}

In particular, the group $G$ is not a direct product group of an infinite group and a non-amenable group. Word hyperbolic groups are typical
examples of $\mathcal{S}$ groups.
Adams \cite{Adams; Indecomposability} showed a measurable indecomposability of
non-amenable word hyperbolic groups. The above corollary covers some part of Adams' theorem.

The class $\mathcal{C}$ of \cite{Monod--Shalom} also contains non-amenable word-hyperbolic groups. A group $G \in \mathcal{C}$ satisfies an indecomposability property; $G$ has no infinite normal amenable subgroup. On the other hand, a non-amenable class $\mathcal{S}$ group can have an infinite normal amenable subgroup (for example, $\mathbb{Z}^2  \rtimes \mathrm{SL}(2, \mathbb{Z}) \in \mathcal{S}$ \cite{Ozawa; An Example}).

\begin{acknowledgment}
This paper was written during the author's stay in UCLA.
The author is grateful to Professor Sorin Popa and Professor Narutaka Ozawa
for their encouragement and fruitful conversations.
He is supported by JSPS Research Fellowships for Young Scientists.
\end{acknowledgment}


\begin{thebibliography}{99}

\bibitem[Ad]{Adams; Indecomposability}
S. Adams,
{\em Indecomposability of equivalence relations generated by word hyperbolic groups},
Topology {\bf 33} (1994), no. 4, 785--798.

\bibitem[AD]{Anantharaman-Delaroche}
C. Anantharaman-Delaroche,
{\em Syst\`{e}mes dynamiques non
commutatifs et moyennabilit\'{e},} Math. Ann. {\bf 279} (1987), no.
2, 297--315.

\bibitem[Ar]{Arveson}
W. Arveson,
{\em Notes on extensions of $C\sp{\sp*}$-algebras,}
Duke Math.~J.~{\bf 44} (1977), no.~2, 329--355.

\bibitem[BrOz]{Brown--Ozawa; Approximation}
N.P. Brown and N. Ozawa,
{\em $C\sp *$-algebras and finite-dimensional approximations,}
Graduate Studies in Mathematics, {\bf 88}. American Mathematical Society, Providence, RI, 2008.

\bibitem[ChEf]{Choi--Effros}
M.D. Choi, E.D. Effros,
{\em The completely positive lifting problem for $C\sp*$-algebras,}
Ann.~of Math.~(2) {\bf 104} (1976), no.~3, 585--609.

\bibitem[CoFeWe]{Connes--Feldman--Weiss}
A. Connes, J. Feldman and B. Weiss, {\em An amenable equivalence
relation is generated by a single transformation,} Ergodic Theory
Dynamical Systems {\bf 1} (1981), no.~4, 431--450 (1982).

\bibitem[FeMoo]{Feldman--Moore; II}
J. Feldman and C.C. Moore,
{\em Ergodic equivalence relations, cohomology, and von
Neumann algebras. II,} Trans.~Amer.~Math.~Soc.~{\bf 234} (1977), no.~2, 325--359.

\bibitem[Fu1]{Furman; Higher Rank Lattice}
A. Furman,
{\em Gromov's measure equivalence and rigidity of higher rank lattices},
Ann.~of Math.~(2) {\bf 150} (1999), no.~3, 1059--1081.

\bibitem[Fu2]{Furman; OE rigidity}
A. Furman,
{\em Orbit equivalence rigidity},
Ann.~of Math.~(2) {\bf 150} (1999), no.~3, 1059--1081.

\bibitem[Gr]{gromov}
M. Gromov,
{\em Asymptotic invariants of infinite groups. Geometric group theory, Vol. 2}, 1--295,
London Math.~Soc.~Lecture Note Ser., {\bf 182}, Cambridge Univ.~Press, Cambridge, 1993.

\bibitem[Hj]{hjorth; cost}
G. Hjorth,
{\em A lemma for cost attained,}
Ann. Pure Appl.~Logic {\bf 143} (2006), no.~1-3, 87--102.

\bibitem[Kid]{Kida}
Y. Kida,
{\em Measure equivalence rigidity of the mapping class group},
to appear in Ann.~of Math.

\bibitem[Kir]{Kirchberg; Weak Exactness}
E. Kirchberg,
{\em Commutants of unitaries in UHF algebras and functorial properties of exactness},
 J.~Reine Angew.~Math.~{\bf 452} (1994), 39--77.

\bibitem[MoSh]{Monod--Shalom}
N. Monod and Y. Shalom,
{\em Orbit equivalence rigidity and bounded cohomology,}
Ann.~of Math.~(2) {\bf 164} (2006), no.~3, 825--878.

\bibitem[MvN]{Murray--vN; IV}
F.J. Murray and
J. von Neumann, {\em On rings of operators. IV,} Ann.~of Math.~(2) {\bf } 44 (1936), 116-229

\bibitem[OrWe]{Ornstein--Weiss}
D. S. Ornstein and B. Weiss, {\em Ergodic theory of amenable group
actions. I. The Rohlin lemma,} Bull.~Amer.~Math.~Soc.~(N.S.) {\bf 2}
(1980), no.~1, 161--164.

\bibitem[Oz1]{ozawa; amenable actions and exactness}
N. Ozawa,
{\em Amenable actions and exactness for discrete groups},
C. R. Acad.~Sci.~Paris S\'{e}r.~I Math.~{\bf 330} (2000), no.~8, 691--695.

\bibitem[Oz2]{ozawa; solid}
N. Ozawa,
{\em Solid von Neumann algebras,} Acta Math.~{\bf 192}
(2004), no.~1, 111--117.

\bibitem[Oz3]{Ozawa; Kurosh}
N. Ozawa,
{\em A Kurosh-type theorem for type $\rm II\sb 1$
factors,} Int.~Math.~Res.~Not. (2006), Art. ID 97560, 21 pages, DOI 10.1155/IMRN/2006/97560.

\bibitem[Oz4]{Ozawa; Weakly Exact von Neumann algebras}
N. Ozawa,
{\em Weakly exact von Neumann algebras},
J.~Math.~Soc.~Japan {\bf 59} (2007), no.~4, 985--991.

\bibitem[Oz5]{Ozawa; An Example}
N. Ozawa, {\em An example of a solid von Neumann algebra},
preprint, arXiv:0804.0288.

\bibitem[OzPo1]{Ozawa--Popa; Prime Factorization}
N. Ozawa and S. Popa,
{\em Some prime factorization results for type
${\rm II}\sb 1$ factors,} Invent.~Math.~{\bf 156} (2004), no.~2,
223--234.

\bibitem[PemPer]{pemantle--peres; treeable}
R. Pemantle and Y. Peres,
{\em Nonamenable products are not treeable,}
Israel J.~Math.~{\bf 118} (2000), 147--155.

\bibitem[Tak]{Takesaki}
M. Takesaki,
{\em On the cross-norm of the direct product of $C\sp{\ast} $-algebras,}
Tohoku Math.~J.~(2) {\bf 16} (1964), 111--122.

\end{thebibliography}
\end{document}